\documentclass[sn-mathphys]{sn-jnl}

\jyear{2021}%

\theoremstyle{thmstyleone}
\newtheorem{theorem}{Theorem}
\newtheorem{proposition}[theorem]{Proposition}%
\newtheorem{lemma}{Lemma}%
\newtheorem{corollary}{Corollary}%
\newtheorem{conjecture}{Conjecture}
\newtheorem{conditional proposition}{Conditional Proposition}

\theoremstyle{thmstyletwo}
\newtheorem{example}{Example}%
\newtheorem{remark}{Remark}%

\theoremstyle{thmstylethree}

\usepackage{graphicx}
\usepackage{soul}
\usepackage[symbol]{footmisc}
\usepackage{xcolor}
\usepackage{mathtools}

\newcommand{\zero}{\mathbf{0}}
\newcommand{\one}{\mathbf{1}}
\renewcommand{\sub}{\sqsubseteq}

\newcommand{\abs}[1]{\lvert#1\rvert}

\newcommand{\edit}[1]{#1}

\begin{document}

\title[A subsemigroup of the rook monoid]{A subsemigroup of the rook monoid}

\author*[1]{George Fikioris}\email{gfiki@ece.ntua.gr}
\author[2]{Giannis Fikioris}\email{gfikioris@cs.cornell.edu}

\affil*[1]{\orgdiv{School of Electrical and Computer Engineering}, \orgname{National Technical University of Athens}, \orgaddress{\street{Zografou}, \city{Athens}, \postcode{GR 15773}, \state{Attica}, \country{Greece}}}

\affil[2]{\orgdiv{Department of Computer Science}, \orgname{Cornell University}, \orgaddress{\street{Hoy Rd}, \city{Ithaca}, \postcode{NY 14853}, \state{New York}, \country{USA}}}

\abstract{
\unboldmath{}
We define a subsemigroup $S_n$ of the rook monoid $R_n$ and investigate its properties. To do this, we represent the nonzero elements of $S_n$ (which are $n\times n$ matrices) via certain triplets of integers, and develop a closed-form expression representing the product of two elements; these tools facilitate straightforward deductions of a great number of properties. For example, we show that $S_n$ consists solely of idempotents and nilpotents, find the numbers of idempotents and nilpotents, compute nilpotency indexes, \edit{determine Green's relations and ideals, and come up with a minimal generating set.}  Furthermore, we give a necessary and sufficient condition for the $j$th root of a nonzero element to exist in $S_n$, show that existence implies uniqueness, and  compute the said root explicitly. We also point to several combinatorial aspects; describe a number of subsemigroups of $S_n$ \edit{(some of which are familiar from previous studies)}; and, using rook $n$-diagrams, graphically interpret many of our results. 
}


\maketitle


\bmhead{Acknowledgments}

\edit{We thank the Reviewer, whose comments and suggestions notably improved this paper. Theorem~\ref{th:minimal-generating-set-for-mn}, in particular, is due to the Reviewer.} The work of Giannis Fikioris was supported in part by \edit{National Science Foundation (NSF) grant CCF-1408673 and Air Force Office of Scientific Research (AFOSR)} grant FA9550-19-1-0183.

\section{Introduction}

For $n=2,3,4,\ldots$, the \textsl{rook monoid} $R_n$, also known as the symmetric inverse semigroup $\mathcal{IS}_n$, consists of all partial injective transformations of $\{1,2,\ldots,n\}$. Any such transformation can be represented as an $n\times n$ matrix whose entries are $0$~or~$1$, and with at most one $1$ in every row and every column. If we think of the matrix as an $n\times n$ chessboard with rooks placed at the nonzero matrix entries, then no rook attacks any other rook; hence the name rook monoid, coined in 2002 by Solomon~\cite{Solomon, Li-Li-Cao}.

This paper introduces a subsemigroup $S_n$ of $R_n$ and studies its properties (note that the symbol $S_n$ does not denote the permutation group). Our $S_n$ is also a subsemigroup of $\mathcal{POI}_n$, i.e., the set of all partial order-preserving injective transformations \cite{Fernandes,East2006,East2010}, which is itself a subsemigroup of $R_n$. Besides $S_n$, we discuss two closely related semigroups, namely the monoid $M_n=S_n\cup \{\one\}$ (where $\one$ is the monoid identity, represented by the $n\times n$ identity matrix), as well as a semigroup of countably infinite order to be denoted by $S_\infty$.

We refer to a number of well-known properties of $R_n$. They are summarized in the lemma that follows.
\begin{lemma}\cite{Ganyushkin}
\label{lemma:rook-monoid}
\begin{itemize}
\item The order of $R_n$ is given by $\abs{R_n}=\sum_{k=0}^nk!\binom{n}{k}^2$.

\item $R_n$ contains $2^n$ idempotents, namely the diagonal matrices belonging to $R_n$. 

\item $R_n$ contains $\sum_{k=1}^n\frac{n!}{k!}\binom{n-1}{k-1}$ nilpotents.

\item $R_n$ is an inverse semigroup, i.e., any $x\in R_n$ has a unique inverse $y\in R_n$. The matrices representing $x$ and $y$ are transposes of one another. 

\item The element $x\in R_n$ has an inverse in the sense of groups (i.e., there exists a $y\in R_n$ such that $xy=yx=\one$) iff $x$ is a bijection, so that it is represented by a permutation matrix.
\end{itemize}
\end{lemma}
\begin{example} 
\label{ex:r2}
For $n=2$, the elements of $R_n=R_2$ can be represented by the following $7=\abs{R_2}$ matrices,
\begin{align*}
\begin{pmatrix}
0 & 1\\
1 & 0
\end{pmatrix},
&\begin{pmatrix}
1 & 0\\
0 & 1
\end{pmatrix}=\one,
\begin{pmatrix}
1 & 0\\
0 & 0
\end{pmatrix}=e,
\begin{pmatrix}
0 & 0\\
0 & 1
\end{pmatrix}=f,\\
&\begin{pmatrix}
0 & 0\\
0 & 0
\end{pmatrix}=\zero,
\begin{pmatrix}
0 & 1\\
0 & 0
\end{pmatrix}=a,
\begin{pmatrix}
0 & 0\\
1 & 0
\end{pmatrix}=b
\end{align*} 
(see Proposition \ref{prop:brandt} for the meaning of the symbols $e$, $f$, $a$, and $b$). 
The group of permutations of $\{1,2\}$ is represented by the first two matrices in our list, which are permutation matrices and represent the only two elements that have an inverse in the sense of groups. The first matrix in our list---to which we assigned no symbol---is a square root of $\one$ ($j$th roots in $R_n$ are discussed in \cite{annin,york}). The next $4=2^2$ matrices are idempotent, and the last three are nilpotent with indexes 1, 2, and 2, respectively. \end{example}

In Section~\ref{section:graphical}, we use the usual maps---called rook $n$-diagrams \cite{xiao}---to give graphical interpretations of some of the properties of $M_n$ and $S_n$. The rook $n$-diagram of any element in the rook monoid $R_n$ is a graph with two rows. Each row has $n$ vertices labeled $1$, $2$, \dots, $n$. The graph's edges start from the top row and end in the bottom one; an edge from $i$ to $j$ means that the element maps $i$ to $j$, so that the matrix entry $x_{ij}$ equals $1$. Each vertex---whether in the top or bottom row---must belong to at most one edge, with the zero matrix having no edges. See the example in Figure~\ref{fig:rook}. 

\begin{figure}[h]
  \centering
  \includegraphics[width=.45\textwidth]{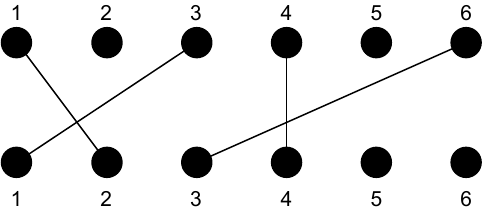}
  \caption{Rook $n$-diagram  of an element of $R_n$. We specifically depict the rook $6$-diagram of the matrix of $R_6$ whose nonzero entries are $x_{12}=x_{31}=x_{44}=x_{63}=1$}
  \label{fig:rook}
\end{figure}

In particular, a rook $n$-diagram corresponds to an element of the submonoid $\mathcal{POI}_n$ iff its edges do not intersect, such as the example in Figure~\ref{fig:poi}.

\begin{figure}[h]
  \centering
  \includegraphics[width=.45\textwidth]{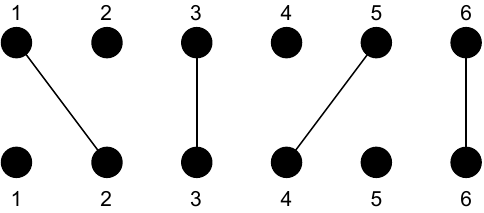}
  \caption{When no edges intersect, a rook $n$-diagram corresponds to an element of $\mathcal{POI}_n$. Here we specifically depict the matrix of  $\mathcal{POI}_6$ whose nonzero entries are $x_{12}=x_{33}=x_{54}=x_{66}=1$}
  \label{fig:poi}
\end{figure}

Concatenation of rook $n$-diagrams corresponds to multiplication, as we will further discuss (for the case of $M_n$) in Section~\ref{section:graphical}.

\section{Notation and conventions}
\label{section:notation}

Throughout, $\zero$ denotes semigroup zeros and $\one$ denotes monoid identities, as in Example~\ref{ex:r2}. The symbol $\abs S$ denotes the \textsl{order} of the semigroup $S$, i.e., the cardinal number of its set of elements. All our semigroups have at least two elements, so that $\zero\ne\one$. A $j$th \textsl{root} of $x\in S$ is a $y\in S$ such that $x^j= y$ ($j\in\mathbb{N}$).  We use $\sub$ in place of ``is a subsemigroup of.'' If $\one$ is the identity of $M_b$ and $\one\in M_a\sub M_b$, we say that $M_a$ is a \textsl{submonoid} of $M_b$ (in Section~\ref{section:submonoids}, we  encounter monoid subsemigroups that are not submonoids because the identities differ).  The term \textsl{inverse element} has its familiar meaning within the context of semigroups; we also use the term \textsl{inverse in the sense of groups}, which we defined in Lemma \ref{lemma:rook-monoid}. \edit{In Section~\ref{section:green}, we use the the traditional notations (e.g., \cite{Lawson}) for Green's relations, associated equivalence classes, and principal ideals.}

The elements of our semigroups are matrices whose entries are $0$ and $1$, called zero and one, respectively. $\mathbb{Z}$ is the set of integers, $\mathbb{N}$ is the set of positive integers, and  $\mathbb{N}_0$ is the set of nonnegative integers. An \textsl{infinite} matrix is one whose  entries are $x_{i,j}$ with $i,j\in \mathbb{N}$. We use $\delta_{ij}$ ($i,j\in \mathbb{N}_0$) to denote the Kronecker delta.   $\lceil x \rceil$ stands for the ceiling of $x\in\mathbb{R}$.

We often assume that elements of isomorphic semigroups coincide. For example, our $M_n$ is a monoid of $n\times n$ matrices. After establishing a multiplication-preserving, one-to-one correspondence between the nonzero matrices of $M_n$ and certain triplets of integers, we make no distinction between a matrix and its representing triplet. As another example, any multiplicative semigroup consisting of $n\times n$ matrices is isomorphic to a semigroup consisting of square matrices whose dimension is larger than $n$, including infinite matrices: We can regard each $n\times n$ matrix as being a submatrix of the larger one, with the remaining entries of the larger matrix all equal to zero. It is in this sense that---in Section~\ref{section:finite}---we define $M_n$ as a subsemigroup of a certain semigroup $S_\infty$, whose elements are infinite matrices. 

While we discuss $M_n$ and $S_n$ primarily via the aforementioned triplets (not the corresponding matrices), our  proof of Theorem~\ref{th:multiplication-inf} makes use of infinite-length row and column vectors whose entries are all $0$ except for at most one entry, which equals $1$. In the said proof, the relevant notation is as follows. $V_q$ ($q\in\mathbb{N}_0$) is the infinite column vector whose entries $V_{qj}$ are
\begin{equation}
    V_{qj}=\delta_{qj}, \quad q\in \mathbb{N}_0, \quad j\in \mathbb{N}
\end{equation}
and $V^T_q$ is the transpose of $V_q$. Therefore, $V^T_q$ is an infinite row vector and  $V_0$ and $V_0^T$ are the infinite zero vectors. Furthermore, the inner product $V^T_qV_p$ is well defined (there is a finite number of nonzero terms, so there are no convergence issues associated with the infinite summation) and is~given~by
\begin{equation}
\label{eq:vectors-inner-product}
    V^T_qV_p=\begin{cases}
    0, \quad p=q=0 \quad \mathrm{or}\quad  p\ne q,\\
    1,\quad p=q\in\mathbb{N}.
    \end{cases}
\end{equation}

\edit{Let $S$ be a finite semigroup, let $A\subseteq S$ be a generating set of $S$, and denote the cardinality of $A$ by $\abs{A}$. $A$ is a \textsl{minimal generating set} (texts such as \cite{Ganyushkin} use \textsl{irreducible} in place of \textsl{minimal}) if no proper subset of $A$ can generate $S$. The \textsl{rank} of $S$, denoted throughout by $\mathrm{rank}(S)$, is defined by \cite{gray2014}
\begin{equation}
    \mathrm{rank}(S)=\mathrm{min}\{\abs{A}:A\subseteq S \mathrm{\ and\ } A \mathrm{\ generates\ } S\}.
\end{equation}
This notation should not be confused with the symbol $\mathrm{rnk}$, used (in Section \ref{section:green} only) for the rank of a partial transformation.
}
\section{Semigroup \texorpdfstring{$S_\infty$}{S infinity} of countably infinite order}
\label{section:infinite}

The elements $x$ of $S_\infty$ are infinite matrices, see the example in Fig. \ref{fig:matrix_example}. Element $\zero\in S_\infty$ is the zero matrix. The nonzero matrices of $S_\infty$ are those for which

\begin{itemize}
    \item[(i)] Each matrix entry is zero ($0$) or one ($1$).
    \item [(ii)] There is a finite and nonzero number of ones, with all ones lying on a single diagonal.
    \item[(iii)] Within this diagonal, the ones form an uninterrupted block, i.e., there is no zero between any two ones.
\end{itemize}

\begin{figure}[h]
    \centering
    \includegraphics{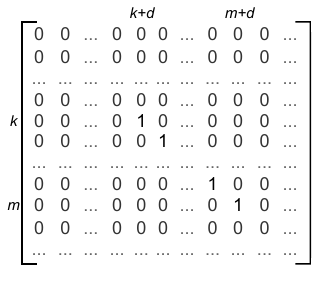}
    \caption{The infinite matrix corresponding to $\langle d,k,m\rangle$}
    \label{fig:matrix_example}
\end{figure}

For any $x\in S_\infty\setminus\{\zero\}$, let $k \in\mathbb{N}$ be the row of the northwestern one; let $m \in \mathbb{N}$  be the row of the southeastern one; and let $d\in\mathbb{Z}$ be the diagonal on which the ones lie, with $d=0$ corresponding to the main diagonal, and $d>0$ ($d<0$) to diagonals above (below) the main one. Then the matrix entries $x_{ij}$ of $x$ are given by
\begin{equation}
\label{eq:elements-infinite}
x_{ij}=
\begin{rcases}
\begin{dcases}
1,\quad k\le i\le m \quad \mathrm{and}\quad  j-i=d,\\
0,\quad \mathrm{otherwise,}
\end{dcases}
\end{rcases},\quad
i,j\in\mathbb{N}.
\end{equation}
Conversely, as long as
\begin{equation*}
1-\min(0,d)\le k\le m, 
\end{equation*}
we can use a triplet of integers $d,k,m$ to represent each $x\in S_\infty\setminus\{\zero\}$. We do this by letting $\langle d,k,m\rangle$ stand for the matrix $x\in S_\infty\setminus \{\zero\}$ whose entries $x_{ij}$ are given by (\ref{eq:elements-infinite}). We can thus consider the set of elements of $S_\infty$ to be 
\begin{equation}
\label{eq:infinite-s-definition}
S_\infty=\{\zero\}\cup\{\langle d,k,m\rangle:\  d\in \mathbb{Z};\  k, m\in\mathbb{N}; \ 1-\min(0,d)\le k \le m\}.
\end{equation}
Note that the number of ones in the nonzero matrix $\langle d,k,m\rangle$ is $m-k+1>0$.

The operation in $S_\infty$ is matrix multiplication, which is well-defined (in the sense that there are no convergence issues) and associative. The theorem that follows shows that this operation is also closed in $S_\infty$ and, whenever the result is nonzero, gives the triplet representing the product. 

\begin{theorem}
\label{th:multiplication-inf}
$S_\infty$ is a noncommutative semigroup of countably infinite order with $x\zero=\zero x=\zero$ for every $x\in S_\infty$. The product of any two nonzero elements $\langle d,k,m\rangle$ and $\langle d',k',m'\rangle$ is given by 
\begin{equation}
\label{eq:multiplication-infinite}
\langle d,k,m\rangle \langle d',k',m'\rangle=
\begin{cases}
\langle d'',k'',m''\rangle,\quad k'' \le m'',\\\
\zero,\quad k''> m'',
\end{cases}
\end{equation}
in which the parameters $d''$, $k''$, and $m''$ are 
\begin{equation}
\label{eq:ddoubleprime}
d''=d+d',
\end{equation}
\begin{equation}
\label{eq:kdoubleprime}
k''=\max(k,k'-d),
\end{equation}
\begin{equation}
\label{eq:mdoubleprime}
m''=\min(m,m'-d).
\end{equation}
Irrespective of whether $k''\le m''$ or not, the parameters  $d''$ and $k''$ satisfy
\begin{equation}
\label{eq:restrictions-on-kdoubleprime}
    1-\min(0,d'')\le k''.
\end{equation}
\end{theorem}

\begin{proof}
The equality $x \zero=\zero x=\zero$ is obvious. As $\langle d,k,m\rangle$ and $\langle d',k',m'\rangle$ both belong to $S_\infty\setminus\{\zero\}$,  (\ref{eq:infinite-s-definition}) gives
\begin{equation*}
    1-\min(0,d)\le k \quad \mathrm{and}\quad 1-\min(0,d')\le k',
\end{equation*}
from which it is easy to show (e.g. by distinguishing cases depending on the signs of $d$, $d'$, and $d''$) that the $d''$ and $k''$ defined in (\ref{eq:ddoubleprime}) and (\ref{eq:kdoubleprime}) satisfy the similar relation (\ref{eq:restrictions-on-kdoubleprime}).

Each entry $x_{ij}$ of the matrix $x=\langle d,k,m\rangle\langle d',k',m'\rangle$ is given by the inner product
\begin{equation}
\label{eq:proof-inf-1}
x_{ij}=\langle d,k,m\rangle_{i*}\langle d',k',m'\rangle_{*j},\quad i,j\in\mathbb{N},
\end{equation}
where $\langle d,k,m \rangle_{i*}$ and $\langle d,k,m \rangle_{*j}$ denote the $i$'th row and  $j$'th column of $\langle d,k,m \rangle$, respectively (our notation for rows and columns is reminiscent of the one in \cite{Meyer}). In terms of the vectors $V_q$ introduced in Section~\ref{section:notation}, it is apparent from (\ref{eq:elements-infinite}) that 
\begin{equation}
\label{eq:proof-inf-2}
   \langle d,k,m\rangle_{i*}=
   \begin{cases}
   V^T_0, \quad i<k\mathrm{\ or\ } i>m,\\
   V^T_{i+d}, \quad k\le i\le m,
   \end{cases}
\end{equation}
and
\begin{equation}
\label{eq:proof-inf-3}
   \langle d',k',m'\rangle_{*j}=
   \begin{cases}
   V_0, \quad j<k'+d'\mathrm{\ or\ } j>m'+d',\\
   V_{j-d'}, \quad k'+d'\le j\le m'+d'.
   \end{cases}
\end{equation}
We can thus use (\ref{eq:vectors-inner-product}) to find the inner product in (\ref{eq:proof-inf-1}). Specifically, (\ref{eq:vectors-inner-product}) and  (\ref{eq:proof-inf-1})--(\ref{eq:proof-inf-3}) tell us that $x_{ij}=1$ iff
\begin{equation}
\label{eq:proof-inf-4}
i+d=j-d'\in\mathbb{N}\quad \mathrm{and}\quad  k\le i\le m\quad \mathrm{and}\quad k'+d'\le j\le m'+d',
    \end{equation}
while $x_{ij}=0$ in any other case. The set of conditions in (\ref{eq:proof-inf-4}) is equivalent to 
\begin{equation}
\label{eq:proof-inf-5}
k''\le i\le m''\quad \mathrm{and}\quad j=i+d''\quad \mathrm{and}\quad i+d\in\mathbb{N},
\end{equation}
where $d''$, $k''$, and $m''$ are defined in (\ref{eq:ddoubleprime})--(\ref{eq:mdoubleprime}); further, it is easy to show that the condition $i+d\in\mathbb{N}$ in (\ref{eq:proof-inf-5}) is superfluous. We have thus arrived at
\begin{equation}
\label{eq:proof-inf-6}
    x_{ij}=
    \begin{cases}
    1,\quad k''\le i\le m''\mathrm{\ and\ } j=i+d'',\\
    0,\quad\mathrm{otherwise.}
    \end{cases}
\end{equation}
When $k''>m''$, (\ref{eq:proof-inf-6}) gives $x_{ij}=0$ for all $i,j\in\mathbb{N}$, so that $x=\zero$ and we have proved the bottom equality in (\ref{eq:multiplication-infinite}).  On the other hand, when $k''\le m''$ we have $x\ne \zero$; furthermore, comparison of (\ref{eq:proof-inf-6}) and (\ref{eq:restrictions-on-kdoubleprime}) to (\ref{eq:elements-infinite}) and (\ref{eq:infinite-s-definition}) gives the top equality, completing our proof of (\ref{eq:multiplication-infinite}) and of closure in $S_\infty$.
\end{proof}

Let us now use Theorem~\ref{th:multiplication-inf} to verify that $S_\infty$ is not a monoid.

\begin{proposition}
\label{prop:no-one}
The semigroup $S_\infty$ has no $\one$.
\end{proposition}
\begin{proof}
Suppose that $\one=\langle d,k,m \rangle$. Then (\ref{eq:ddoubleprime}) and $\one \langle d',k',m' \rangle=\langle d',k',m' \rangle$ implies $d=0$, so that $\one=\langle 0,k,m \rangle$.
Thus, by Theorem \ref{th:multiplication-inf}, 
\begin{equation*}
    \one \langle 0,k,m+1 \rangle=\langle 0,k,m \rangle\langle 0,k,m+1 \rangle=\langle 0,k,m \rangle,
\end{equation*}
but on the other hand,
\begin{equation*}
    \one \langle 0,k,m+1 \rangle=\langle 0,k,m+1 \rangle,
\end{equation*}
implying the contradiction $m+1=m$.
\end{proof} 

(It is tempting to say that $\one$ is the diagonal matrix with all entries equal to $1$ but the said matrix does not belong to $S_\infty$ by definition.)

Theorem~\ref{th:multiplication-inf} can yield further properties of $S_\infty$. Our main focus, however, is a similar semigroup $M_n$ of \textit{finite} order, which we now proceed to deal with.

\section{Monoid \texorpdfstring{$M_n$}{Mn} and semigroup \texorpdfstring{$S_n$}{Sn}}
\label{section:finite}

Although it is possible to discuss $M_n$ directly, we avoid calculations by introducing $M_n$ as a subsemigroup of $S_\infty$. Specifically, for any fixed $n$ with  $n=2,3,\ldots$ we define
\begin{equation}
    M_n=\{\zero\}\cup \{x\in S_\infty\setminus\{\zero\}: m\le n-\max(0,d)\}. 
\end{equation}
Thus $M_n$ can be regarded as a semigroup of $n\times n$ matrices, in the sense explained in Section~\ref{section:notation}; that is to say, any element of $M_n$ is an $n\times n$ submatrix (on the top left) of the infinite matrix in  Figure~\ref{fig:matrix_example}. 

For any $x\in M_n\setminus\{\zero\}$, properties (i)--(iii) of Section \ref{section:infinite} continue to hold.  Any nonzero $x\in M_n$ has matrix entries $x_{ij}$ that are given by (\ref{eq:elements-infinite}), but with $i,j\in\{1,2,\ldots,n\}$, and with the meaning of the integers $d$, $k$, $m$ remaining the same.
Clearly, $M_n$ is a subsemigroup of the rook monoid~$R_n$.

In terms of triplets, the set of elements of $M_n$ is
\begin{equation}
\label{eq:m-definition}
\begin{split}
M_n&=\{\zero\}\\ 
&\cup
\{\langle d,k,m\rangle:\  d\in \mathbb{Z};\  k, m\in\mathbb{N}; \ 1-\min(0,d)\le k \le m\le n-\max(0,d)\}.
\end{split}
\end{equation}
In addition to $k,m\in\{1,2,\ldots,n\}$, the restrictions written in (\ref{eq:m-definition}) imply
\begin{equation}
\label{eq:restriction-d}
    -(n-1)\le d\le n-1,
\end{equation}
for any $\langle d,k,m\rangle\in M_n\setminus\{\zero\}$; (\ref{eq:restriction-d}) is the range of the diagonal $d$, as expected.

The multiplication formula (\ref{eq:multiplication-infinite}) remains unaltered,  but $M_n$ is associated with an additional restriction on $d''$ and $m''$. Furthermore---and even if the underlying semigroup $S_\infty$ is \textit{not} a monoid (Proposition~ \ref{prop:no-one})---we now have a $\one$ :

\begin{theorem}
\label{th:multiplication-finite}
$M_n$ is a noncommutative monoid with $\one=\langle 0,1,n\rangle$. The product of two nonzero monoid elements $\langle d,k,m\rangle$ and $\langle d',k',m'\rangle$---see (\ref{eq:m-definition})---is given by 
\begin{equation}
\label{eq:multiplication-finite}
\langle d,k,m\rangle \langle d',k',m'\rangle=
\begin{cases}
\langle d'',k'',m''\rangle,\quad k'' \le m'',\\\
\zero,\quad k''> m'',
\end{cases}
\end{equation}
in which the parameters $d''$, $k''$, and $m''$ are 
\begin{equation}
\label{eq:ddoubleprime-finite}
d''=d+d',
\end{equation}
\begin{equation}
\label{eq:kdoubleprime-finite}
k''=\max(k,k'-d),
\end{equation}
\begin{equation}
\label{eq:mdoubleprime-finite}
m''=\min(m,m'-d).
\end{equation}
Irrespective of whether $k''\le m''$ or not, we have
\begin{equation}
\label{eq:restrictions-on-kdoubleprime-finite}
    1-\min(0,d'')\le k''\quad \mathrm{and}\quad m''\le n-\max(0,d'').
\end{equation}
\end{theorem}
\begin{proof}
As $\langle d,k,m\rangle$ and $\langle d',k',m'\rangle$ belong to $M_n\setminus\{\zero\}$, (\ref{eq:m-definition}) gives
\begin{equation*}
    m\le n-\max(0,d)\quad \mathrm{and}\quad m'\le n-\max(0,d'),
\end{equation*}
from which it is easy to show (e.g. by distinguishing cases depending on the signs of $d$, $d'$, and $d''$) that $d''$ and $m''$ satisfy the second inequality in (\ref{eq:restrictions-on-kdoubleprime-finite}).
Since $M_n$ is a subsemigroup of $S_\infty$,  (\ref{eq:multiplication-finite})--(\ref{eq:mdoubleprime-finite}) follow immediately from Theorem~\ref{th:multiplication-inf}, as does the first inequality in (\ref{eq:restrictions-on-kdoubleprime-finite}).  Since $\langle 0,1,n\rangle$ represents the $n\times n$ identity matrix, the equality $\one=\langle 0,1,n\rangle$ is obvious. (Alternatively, we can use (\ref{eq:multiplication-finite})--(\ref{eq:mdoubleprime-finite}) and the restrictions in (\ref{eq:m-definition}) to show that
$x\langle 0,1,n\rangle =\langle 0,1,n\rangle x =x$
for every $x=\langle d,k,m\rangle\in M_n\setminus\{\zero\}$.)
\end{proof}
\begin{remark}
\label{remark:number-of-ones}
Since the number of ones in the matrix $\langle d'',k'',m''\rangle$ equals $m''-k''+1$, the condition $k''\le m''$ in (\ref{eq:multiplication-finite}) assures that the \edit{length of the block of ones} (in any matrix which corresponds to a nonzero product) is positive.
\end{remark}

Reformulating the inequality $k''\le m''$, we obtain alternative conditions for a product to be zero/nonzero.

\begin{corollary} 
\label{corr:nonzero}
Let $x=\langle d,k,m\rangle\in M_n\setminus\{\zero\}$ and $y=\langle d',k',m'\rangle\in M_n\setminus\{\zero\}$, and let $d''$, $k''$, and $m''$ be given by (\ref{eq:ddoubleprime-finite})--(\ref{eq:mdoubleprime-finite}). Then
\begin{equation}
\label{eq:nonzero-product}
xy\ne \zero\iff k''\le m''\iff
k'-m\le d\le m'-k.
\end{equation}
\end{corollary}
\begin{proof}
By (\ref{eq:kdoubleprime-finite}) and (\ref{eq:mdoubleprime-finite}), the $k''\le m''$ in (\ref{eq:multiplication-finite}) is equivalent to the four inequalities
\begin{equation*}
    k\le m\quad\mathrm{and} \quad k\le m'-d\quad\mathrm{and} \quad k'-d\le m\quad\mathrm{and} \quad k'-d\le m'-d.
\end{equation*}
The first and last of the four were known beforehand from (\ref{eq:m-definition}). The second and third give the last condition in (\ref{eq:nonzero-product}).
\end{proof}

A formula for powers, to be used many times in this paper, can be verified using Theorem~\ref{th:multiplication-finite},  Remark~\ref{remark:number-of-ones}, and induction:
\begin{corollary}
\label{corr:powers}
For $x=\langle d,k,m\rangle\in M_n\setminus\{\zero\}$ and $j\in\mathbb{N}$ we have
\begin{equation}
\label{eq:powers-1}
x^j=\begin{cases}
\left\langle d^{(j)},k^{(j)},m^{(j)}\right\rangle,\quad \textrm{if}\quad  k^{(j)}\le m^{(j)},
\\
\zero,\quad \textrm{if}\quad  k^{(j)}>m^{(j)},
\end{cases}
\end{equation}
where
\begin{equation}
\label{eq:powers-2}
d^{(j)}=jd,\quad k^{(j)}=k-(j-1)\min(0,d),\quad m^{(j)}=m-(j-1)\max(0,d).
\end{equation}
Furthermore, as long as $x^j\ne\zero$, the number of ones in the matrix representing $x^j$ equals $m-k+1-(j-1)\abs d$.
\end{corollary}

The elements of $M_n$ are idempotent or nilpotent. This is shown in the theorem that follows, which further gives the nilpotency index.
\begin{theorem}
\label{th:idempotent-nilpotent}
An element $\langle d,k,m \rangle\in M_n\setminus \{\zero\}$ is idempotent if $d=0$ and nilpotent if $d\ne 0$. In the latter case, the index $\ell$ of the nilpotent is given (in terms of $\abs d$ and the number $m-k+1$ of ones in the matrix $\langle d,k,m\rangle$) by
\begin{equation}
\label{eq:index-of-nilpotent}
\ell=1+\left\lceil \frac{m-k+1}{\abs d} \right\rceil,
\end{equation}
where $\lceil x \rceil$ denotes the ceiling of $x\in\mathbb{R}$. This index satisfies $2\le\ell\le n$.
\end{theorem}
\begin{proof}
For $d=0$ and $j=2$, Corollary~\ref{corr:powers} gives 
    $\langle 0,k,m\rangle^2=\langle 0,k,m\rangle$
so that $\langle 0,k,m\rangle$ is idempotent. Assume that $d<0$. Then the $k^{(j)}=k+\abs d(j-1)$ in (\ref{eq:powers-2}) increases with $j$ and will thus surpass the constant $m^{(j)}=m$. By (\ref{eq:powers-1}), this means that  $\langle 0,k,m\rangle^j=\zero$ for sufficiently large $j$, so that $\langle d,k,m\rangle$ is nilpotent. The index $\ell$, which is the smallest $j$ for which $k^{(j)}>m$, is then  given by (\ref{eq:index-of-nilpotent}). As $\langle d,k,m\rangle$ is  not a diagonal matrix, the numerator and denominator of $\frac{m-k+1}{\abs d}$ both lie between $1$ and $n-1$. Thus $2\le\ell\le n$. (More generally, the nilpotency index of \textit{any} nonzero nilpotent $n\times n$ matrix satisfies this relation, see p. 190 of  \cite{HornJohnson}.) The proof for $d>0$ is similar.
\end{proof}
As expected, $\ell$ is independent of $n$. Eqn. (\ref{eq:index-of-nilpotent}) further shows that moving the block of ones along the underlying diagonal leaves the nilpotency index unaltered.

$\one$ is the only permutation matrix in $M_n$, which is a submonoid of $R_n$. It thus follows from Lemma \ref{lemma:rook-monoid} that $\one$ is the only element of $M_n$ which has an inverse in the sense of groups. We can strengthen this statement by the proposition that follows, which is not true for $x,y\in R_n$ (but \textit{is} true for any monoid consisting solely of nilpotents and idempotents).

\begin{proposition}
\label{prop:not-a-group}
Let $x,y\in M_n$. We have $x y=\one$ iff $x=y=\one$.
\end{proposition}
\begin{proof}
Let $x y=\one$, so that $x\ne \zero$. If $x$ is nilpotent with index $\ell$, then $\ell\ge 2$, so left-multiplication of $x y=\one$ by $x^{\ell-1}$ gives $x^\ell y=x^{\ell-1}$, implying $\zero=x^{\ell-1}$ and contradicting the definition of the index. Thus $x$ is not nilpotent. By Theorem~\ref{th:idempotent-nilpotent}, $x$ is idempotent, so $x y=\one$ gives $xy=x$. Thus $x=xy=\one$ and $\one=x y=\one y=y$. 
\end{proof}
 Multiplication is thus closed within $M_n\setminus\{\one\}$. Consequently,
\begin{corollary}
\label{corr:S-semigroup}
$S_n=M_n\setminus\{\one\}$ is a semigroup and is a subsemigroup of $R_n$.
\end{corollary}

Let us note that there is some risk of confusion between four closely related entities dealt with herein; namely, the monoid $M_n=S_n\cup\{\one\}$, the semigroup $S_n=M_n\setminus\{\one\}$, the set $M_n\setminus\{\zero\}$,  and the set $S_n\setminus\{\zero\}$. While sets (consisting of the elements that can be represented by triplets), the last two entities are not semigroups.

We close this section by providing an example that views a familiar semigroup as a subsemigroup of $S_n$.
\begin{example}
\label{ex:mequalsk-example}
Let us specialize  our multiplication formula to elements of the set $B_n$ defined by
 \begin{equation}
\label{eq:b-definition}
B_n=\{\zero\}\cup
\{\langle d,k,m\rangle\in M_n\setminus\{\zero\}:m=k\}.
\end{equation}
With the aid of Corollary \ref{corr:nonzero} (take $m=k$ and $m'=k'$), we see that (\ref{eq:multiplication-finite})--(\ref{eq:mdoubleprime-finite}) reduce to
\begin{equation*}
\label{eq:b-multiplication-finite}
\langle d,k,k\rangle \langle d',k',k'\rangle=
\begin{cases}
\langle d+d',k,k\rangle,\quad k'=k+d,\\
\zero,\quad \mathrm{otherwise}.
\end{cases}
\end{equation*}
Thus multiplication is closed and $B_n$ is a subsemigroup of $S_n$.  Now set $(k,p)=\langle p-k,k,k\rangle$; in this new notation, \edit{the multiplication formula becomes}
\begin{equation}
\label{eq:b-multiplication-finite-new-notation}
\big(k,p\big)\big(k',p'\big)=
\begin{cases}
\big(k,p'\big),\quad k'=p,\\
\zero,\quad \mathrm{otherwise}.
\end{cases}
\end{equation}
For $1\le k,p\le n$, the matrix $(k,p)$ is, by definition, the $n\times n$ matrix $x$ with all elements zero except for element $x_{kp}$, which  equals 1.
Therefore (\ref{eq:b-multiplication-finite-new-notation}) is self-evident and $B_n$ is the well-known combinatorial, $0$-bisimple inverse semigroup associated with Brandt semigroups, see Chapter~3 of \cite{Lawson}. Theorem~\ref{th:idempotent-nilpotent} corroborates that the nonzero elements of $B_n$ are idempotent if $d=0$ ($p=k$) and nilpotent with index $\ell=2$ if $d\ne 0$ ($p\ne k$). We further note that $(k,p)$, which is sometimes called the single-entry matrix, arises in ring theory \cite{Lawson}, as well as other applications \cite{Shimizu}. 
\end{example}

\section{\texorpdfstring{$j$}{j}th roots}

As mentioned in our Introduction, \cite{annin} and \cite{york} discuss $j$th roots in the monoid $R_n$.  In its submonoid $M_n$, the triplet representation renders the problem of finding roots more straightforward: Firstly, the roots of $\zero$ can be computed via Theorem~\ref{th:idempotent-nilpotent}. Secondly, any nonzero idempotent $x$ has a unique $j$th root, namely $x$ itself (as all $y\in M_n$ are idempotent or nilpotent, showing uniqueness is trivial).  Thirdly, the roots of any nonzero nilpotent $x$ can be found from the theorem that follows (the theorem actually holds for any $x\in M_n$, as long as $x\ne \zero$). As it turns out, if a root $y$ of a nonzero $x$ does exist, it is unique.

\begin{theorem}
\label{th:roots}
Let $j\in\mathbb{N}$ and  let $x=\langle d,k,m\rangle\in M_n\setminus\{\zero\}$. Then a $j$th root of $x$ exists in $M_n$ iff
$d$ is an integer multiple of $j$. In this case, the $j$th root is unique in $M_n$ and given by $y=\langle d',k',m'\rangle\in M_n\setminus\{\zero\}$, where
\begin{equation}
\label{eq:root-parameters}
d'=\frac{d}{j},\quad k'=k+(j-1)\min(0,d'),\quad m'=m+(j-1)\max(0,d').
\end{equation}
\end{theorem}
\begin{proof}
We take $d\ge 0$. Then (\ref{eq:m-definition}) and $x=\langle d,k,m\rangle\in M_n\setminus\{\zero\}$ yield
\begin{equation}
\label{eq:root-condition-on-x}
    1\le k\le m\le n-d.
\end{equation}
We seek $y\in M_n$ such that $x=y^j$. Since $y\ne \zero$, we can set $y=\langle d', k', m'\rangle\in M_n\setminus\{\zero\}$.
By Corollary~\ref{corr:powers}, $y$ exists iff there exist integers $d'$, $k'$, and $m'$ satisfying
\begin{equation}
\label{eq:root-condition-1}
d=jd',\quad k=k',\quad m=m'-(j-1)d',
\end{equation}
which, because of (\ref{eq:m-definition}), must also satisfy
\begin{equation}
\label{eq:root-condition-2}
    1\le k'\le m'\le n-d'.
\end{equation}
In (\ref{eq:root-condition-1}) and (\ref{eq:root-condition-2}), we took $d'\ge 0$ because no negative $d'$ can satisfy the first equality (\ref{eq:root-condition-1}). If $d$ is \textit{not} an integer multiple of $j$, then the  first equality (\ref{eq:root-condition-1}) cannot be satisfied, so no root $y$ exists. If $d$ \textit{is} an integer multiple of $j$, we can uniquely solve (\ref{eq:root-condition-1}) for $d'$, $k'$, $m'$; the solution thus obtained is given by (\ref{eq:root-parameters}) (when (\ref{eq:root-parameters}) is specialized to $d>0$). Using (\ref{eq:root-condition-on-x}), it is then easy to show that the aforementioned $d'$, $k'$, $m'$ further satisfy (\ref{eq:root-condition-2}). Accordingly, in this case there is a unique root $y$, whose triplet parameters are given by (\ref{eq:root-parameters}). The proof for $d<0$ is similar.
\end{proof}
\begin{remark}
For $d=0$, Theorem~\ref{th:roots} recovers the above-stated result for nonzero idempotents. For nonzero nilpotents and for $j\ge 2$, the theorem and (\ref{eq:restriction-d}) (or, alternatively, the inequality $\ell\le n$ of Theorem~\ref{th:idempotent-nilpotent}) imply that a $j$th root can exist only if $j\le n-1$. For example, $\zero$ is the only nilpotent of $M_2$ that possesses a square root in $M_2$ (in fact, $\zero$ has three square roots in $M_2$, namely the nilpotents of $M_2$, see Example~\ref{ex:r2} or   Proposition~\ref{prop:brandt}). Needless to say, Theorem~\ref{th:roots} does not pertain to roots in $R_n\setminus M_n$, such as the additional square root of $\one\in M_2$ that we noted in Example~\ref{ex:r2}.
\end{remark}

\section{Further properties of \texorpdfstring{$M_n$}{Mn} and \texorpdfstring{$S_n$}{Sn}}

The triplet representation and simple combinatorics allow us to find the number of elements in $S_n$ and $M_n$. As expected, $\abs{M_n}$ is much smaller than $\abs{R_n}$ (see Lemma~\ref{lemma:rook-monoid}) when $n$ is large:

\begin{proposition}
\label{prop:order}
The order $\abs{S_n}$ of $S_n$ equals the square pyramidal number \edit{(sequence A000330 of the OEIS \cite{oeis})}, while $\abs{M_n}=\abs{S_n}+1$. In other words,
\begin{equation}
\label{eq:order}
\abs{S_n}=\abs{M_n}-1
=\frac{n^3}{3}+\frac{n^2}{2}+\frac{n}{6}.
\end{equation}
\end{proposition}
\begin{proof}
Each nonzero element of $M_n$ uniquely defines the northwest/southeast diagonal of a square and vice versa. Therefore $\abs{S_n}=\abs{M_n}-1$ equals the number of squares within an $n\times n$ square grid, which is \cite{oeis} the square pyramidal number written in the right-hand side of (\ref{eq:order}). Alternatively, (\ref{eq:restriction-d}) and (\ref{eq:m-definition})
give
\begin{equation*}
\abs{M_n}-1=\sum_{d=-(n-1)}^{n-1}\,
\sum_{k=1-\min(0,d)}^{n-\max(0,d)}\,
\sum_{m=k}^{n-\max(0,d)}1,
\end{equation*}
from which (\ref{eq:order}) follows by direct summation.
\end{proof}
\begin{example}
 For $n=2$, $M_n=M_2$ consists of all matrices in Example \ref{ex:r2} with the exception of the first matrix $\begin{pmatrix}
0 & 1\\
1 & 0
\end{pmatrix}$ (to which we assigned no symbol), so that  $\abs{M_2}=6$ and $\abs{S_2}=5$, as given by (\ref{eq:order}). 
\end{example}

The semigroup $S_2$ has in fact been well studied \cite{Howie,Brandt}:

\begin{proposition}
\label{prop:brandt}
The Cayley table of $S_2$ is \edit{depicted in Table~\ref{tab:cayley}}. \edit{Furthermore,}
$S_2$ is isomorphic to $B_2$, where $B_2$ is the five-element Brandt semigroup. 
\begin{table}[!h]
\centering
\begin{tabular}{ c| c c c c c}
 {} & $\zero$ & $a$ & $b$ & $e$ & $f$\\ 
 \hline
$\zero$ & $\zero$ & $\zero$ & $\zero$ & $\zero$ & $\zero$\\ 
$a$ & $\zero$ & $\zero$ & $e$ & $\zero$ & $a$\\ 
$b$ & $\zero$ & $f$ & $\zero$ & $b$ & $\zero$\\ 
$e$ & $\zero$ & $a$ & $\zero$ & $e$ & $\zero$\\ 
$f$ & $\zero$ & $\zero$ & $b$ & $\zero$ & $f$\\ 
\end{tabular}
\caption{\edit{The Cayley table of $S_2$. The symbols $e$, $f$, $a$, $b$ are those used in Example \ref{ex:r2}}}
\label{tab:cayley}
\end{table}
\end{proposition}
\begin{proof} For $n=2$, set $e=\langle 0,1,1\rangle$, $f=\langle 0,2,2\rangle$, $a=\langle 1,1,1\rangle$, and $b=\langle -1,2,2\rangle$. The table given above can then be verified using Theorem~\ref{th:multiplication-finite} (or by multiplying the corresponding matrices), and coincides with the table of $B_2$ in p. 32 of \cite{Howie}. The isomorphism is also apparent from our description of $B_n$ in Example~\ref{ex:mequalsk-example}.  
\end{proof}

As seen from its Cayley table, $S_2$ has 3 idempotents (namely, $\zero$, $e$, and $f$) and 3 nilpotents ($\zero$, $a$, and $b$). More generally, the number of idempotents is the frequently-arising \cite{East2010} triangular number $\frac{n(n+1)}{2}$ (sequence A000217 of the OEIS \cite{oeis}):

\begin{proposition}
\label{prop:band}
$S_n=M_n\setminus \{\one\}$ has $\frac{n(n+1)}{2}$ idempotents and $\frac{n^3}{3}-\frac{n}{3}+1$ nilpotents. 
\end{proposition}
\begin{proof}
Let $N$ be the number of idempotents in $S_n=M_n\setminus\{\one\}$. Then $N$ equals the number of idempotents in $M_n\setminus\{\zero\}$, which is in turn equal to the number of elements $\langle 0,k,m\rangle$. By Theorem~\ref{th:idempotent-nilpotent}, this number is 
\begin{equation*}
N=\sum_{k=1}^{n}\,
\sum_{m=k}^{n}1=\frac{n(n+1)}{2},
\end{equation*}
a result which can also be found using elementary combinatorics. The number of nilpotents, which is $\abs{S_n}-N+1$, then follows from Proposition \ref{prop:order}.
\end{proof}

Proposition \ref{prop:band} tells us that $R_n$ contains more idempotents than does $M_n$, see Lemma \ref{lemma:rook-monoid}. As for nilpotents, all those of $R_2$ belong to $M_2$, but this is no longer true when $n\ne 2$.

In $R_n$, any two inverse elements are transposes of \edit{one another, in the usual sense of matrix transposition. Now $\langle d,k,m\rangle^T$ belongs to $M_n$ and its triplet parameters are obvious. We have thus} arrived at the following proposition. 

\begin{proposition} 
\label{prop:transpose}
Let  $x\in M_n$. The transpose $x^T$ of $x$ is
\begin{equation}
\label{eq:transpose}
x^T=
\begin{cases}
\zero,\quad x=\zero,\\
\langle -d,k+d,m+d\rangle, \quad x=\langle d,k,m \rangle\in M_n\setminus\{\zero\}.
\end{cases}
\end{equation}
\edit{$x^T$ belongs to $M_n$ and is the unique inverse of $x$. Therefore $M_n$ is an inverse monoid and $S_n=M_n\setminus\{\one\}$} is an inverse semigroup.
\end{proposition}

\edit{The corollary that follows will help us in the next section, where we determine Green's relations.}
\edit{
\begin{corollary}
\label{corr:xtimesinversex}
In $M_n$, the products $xx^T$ and $x^Tx$ of an element with its inverse are given by 
\begin{equation}
\label{eq:productxtimesinversex}
xx^T=
\begin{cases}
\zero,\quad x=\zero,\\
\langle 0,k,m\rangle, \quad x=\langle d,k,m \rangle\in M_n\setminus\{\zero\},
\end{cases}
\end{equation}
\begin{equation}
\label{eq:productinversextimesx}
x^Tx=
\begin{cases}
\zero,\quad x=\zero,\\
\langle 0,k+d,m+d\rangle, \quad x=\langle d,k,m \rangle\in M_n\setminus\{\zero\}. 
\end{cases}
\end{equation}
\end{corollary}}

While $R_n$ and $S_n$ are not abelian semigroups, some nonzero elements do commute, e.g., $e$ and $f$ of $S_2$, see Proposition \ref{prop:brandt}. What follows is a necessary and sufficient condition for two nonzero $x,y\in M_n$ to commute ($xy=yx$).
\begin{proposition}
\label{prop:commute-condition}
Let $x,y\in M_n\setminus\{\zero\}$ with $x=\langle d,k,m\rangle$ and $y=\langle d',k',m'\rangle$. Then
\begin{itemize}
    \item Case 1:
$xy=yx\ne\zero$ iff 
\begin{equation}
\label{eq:commute-condition-nonzero}
\max(k,k'-d)=\max(k',k-d')\le \min(m,m'-d)=\min(m',m-d').
\end{equation}
\item Case 2:
$xy=yx=\zero$ iff 
\begin{equation}
\label{eq:commute-condition-zero}
\max(k,k'-d)>\min(m,m'-d)\quad\mathrm{and}\quad    \max(k',k-d')>\min(m',m-d').
\end{equation}
\end{itemize}
\end{proposition}
\begin{proof}
Apply Corollary~\ref{corr:nonzero}, interchange the roles of $x$ and $y$, and apply again.
\end{proof}
In the special case $d=d'=0$, the conditions of either Case~1 or Case~2 are satisfied for all $k,m,k',m'$ (as expected, because diagonal matrices commute, see also Section~\ref{section:mdn}).  Thus the idempotents of $M_n$ and $S_n$ commute, which is tantamount to saying that $M_n$ and $S_n$ are orthodox semigroups. This was known beforehand, because all inverse semigroups are orthodox.
 
\section{Green's relations on \texorpdfstring{$M_n$}{Mn}; Ideals of \texorpdfstring{$M_n$}{Mn}}
\label{section:green}
\edit{The theorem below gives Green's relations on $M_n$. For a matrix $x\in M_n$, it is helpful to use $\mathrm{rnk}(x)$ to denote the length of the block of ones, which equals the rank of the corresponding partial injective transformation. We thus define
\begin{equation}
\label{eq:rand}
\mathrm{rnk}(x)=
\begin{cases}
0,\quad x=\zero,\\
m-k+1, \quad x=\langle d,k,m \rangle\in M_n\setminus\{\zero\},
\end{cases}
\end{equation}
with $0\le\mathrm{rnk}(x)\le n$. By Theorem~\ref{th:multiplication-finite} (or by the identical result in $R_n$\cite{Ganyushkin}),
\begin{equation}
\label{eq:rank-property}
    \mathrm{rnk}(xy)\le \min\left(\mathrm{rnk}(x),\mathrm{rnk}(y)\right),\quad x,y\in M_n.
\end{equation}
}
\edit{
\begin{theorem}
\label{th:green}
    In the inverse monoid $M_n$, Green's relations for any two nonzero elements $x=\langle d,k,m\rangle$ and $y=\langle d',k',m'\rangle$ are as follows.
    \begin{equation}
    \label{eq:green-R}
        x\mathcal{R}y\iff k=k'\ \mathrm{and}\ m=m',
    \end{equation}
    \begin{equation}
    \label{eq:green-L}
       x\mathcal{L}y\iff k+d=k'+d'\ \mathrm{and}\ m+d=m'+d',
    \end{equation}
    \begin{equation}
    \label{eq:green-H}
        x\mathcal{H}y\iff\ x=y,
    \end{equation}
    \begin{equation}
    \label{eq:green-D-or-J}
        x\mathcal{D}y\iff x\mathcal{J}y\iff\ m-k=m'-k'\iff \mathrm{rnk}(x)=\mathrm{rnk}(y).
    \end{equation}
    In all cases, $\zero$ forms a class of its own,
    \begin{equation}
    \label{eq:green-zero}
    R_\zero=L_\zero=H_\zero=D_\zero=J_\zero=\{\zero\}.
    \end{equation}
\end{theorem}
}
\begin{proof}
    \edit{By the definitions  of $\mathcal{R}$ and $\mathcal{L}$ in inverse semigroups \cite{Lawson},  $x\mathcal{R}y$ iff $xx^T=yy^T$, and $x\mathcal{L}y$ iff $x^Tx=y^Ty$. Thus (\ref{eq:green-R}) and $R_\zero=\{\zero\}$ follow from (\ref{eq:productxtimesinversex}), while (\ref{eq:green-L}) and $L_\zero=\{\zero\}$ follow from  (\ref{eq:productinversextimesx}). The definition $\mathcal{H}=\mathcal{R}\cap \mathcal{L}$ then implies $H_\zero=\{\zero\}$. It further implies that $x\mathcal{H} y$ iff the conditions in  (\ref{eq:green-R}) and (\ref{eq:green-L}) hold simultaneously.  Thus $d=d'$ which, together with (\ref{eq:green-R}), gives (\ref{eq:green-H}).} 
    
    \edit{
    Alternatively, (\ref{eq:green-R})--(\ref{eq:green-H}) (as well as the pertinent relations in (\ref{eq:green-zero})) can be shown via the corresponding Green's relations on the inverse semigroup $R_n$ \cite{Ganyushkin,Lawson}, which are inherited \cite{Howie,Lawson}  by the inverse subsemigroup $M_n$.}
    
    \edit{We now turn to the condition for $\mathcal{D}$ in  (\ref{eq:green-D-or-J}). 
    As long as $m-k=m'-k'$, we can define a triplet $z=\langle d'',k'',m''\rangle$ by 
    \begin{equation}
    \label{eq:z-construct}
        z=\langle k+d-k',k',m'\rangle=
        \langle m+d-m',k',m'\rangle.
    \end{equation}
     Now write the conditions in (\ref{eq:m-definition}) for $d$, $k$, $m$, and again for $d'$, $k'$, $m'$. Upon invoking $m-k=m'-k'$, we can easily deduce identical conditions for $d''$, $k''$, $m''$. Thus $z\in M_n$. Furthermore, our $z$ satisfies $z\mathcal{R}y$ by (\ref{eq:green-R}) and  $x\mathcal{L}z$ by (\ref{eq:green-L}). By  the definition of $\mathcal{D}$, the existence of such a $z$ means $x\mathcal{D}y$.}
     
     \edit{In the case $x=\zero$, the equalities $R_\zero=L_\zero=\{\zero\}$ tell us that $x\mathcal{L}z$ and $z\mathcal{R}y$ can only hold when $z=y=\zero$. Thus $D_\zero=\zero$.}
     
     \edit{The assertions concerning $\mathcal{J}$ follow from the equality $\mathcal{J}=\mathcal{D}$, which holds because  the semigroup $M_n$ is of finite order \cite{Howie}.}
\end{proof}

\edit{Theorem~\ref{th:green} implies that the inverse monoid $M_n$ is combinatorial  (as $\mathcal{H}$ is the equality relation \cite{Lawson}), and states that $x\mathcal{R}y$ ($x\mathcal{L}y$) amounts to a horizontal (vertical) translation of the block of ones in the matrices representing $x$ and $y$. Furthermore, the $z$ of (\ref{eq:z-construct}) has a block of ones which is, concurrently, a horizontal translation of the block of $y$ and a vertical translation of the block of $x$. This explains (\ref{eq:green-D-or-J}), which states that all $\mathcal{D}$- (or $\mathcal{J}$-) related matrices have blocks of ones with equal lengths, i.e., identical ranks.}

\edit{The theorem that follows (see \cite{Ganyushkin} for similar results pertaining to $R_n$) uses Theorem~\ref{th:green} and (\ref{eq:rank-property}) to describe all ideals of $M_n$.
\begin{theorem}
\label{th:ideals}
$M_n$ has  $n+1$ (two-sided) ideals $I_0, I_1,\ldots ,I_n$, all of which are principal. 
Ideal $I_r$ can be found from
\begin{equation}
\label{eq:ideals-given-by}
    I_r=M_nxM_n=\{y\in M_n:\mathrm{rnk}(y)\le r\},\quad r=0,1,\ldots,n,
\end{equation}
where $x$ is any element of $M_n$ with $\mathrm{rnk}(x)=r$.
The ideals form a chain according to
\begin{equation}
\label{eq:chain}
    \{\zero\}=I_0\subset I_1\subset\ldots\subset I_n=M_n.
\end{equation}
\end{theorem}
}

\begin{proof}
\edit{
The results on $\mathcal{J}$-classes in Theorem~\ref{th:green} imply $M_nxM_n=M_nyM_n$ for any $y$ with $\mathrm{rnk}(y)=\mathrm{rnk}(x)$. Consequently, the principal ideal $M_nxM_n$ is completely specified by $r=\mathrm{rnk}(x)$, justifying the notation $I_r=M_nxM_n$ that we used in (\ref{eq:ideals-given-by}).} 

\edit{
In particular, we can choose a diagonal matrix $y$ whose block of ones starts at the top left. Thus there are $n+1$ principal ideals $I_r$, which can be written as
\begin{equation}
\label{eq:ideals-definition-via-diagonal}
I_r=
\begin{cases}
\{\zero\},\quad r=0,\\
M_n\langle 0,1,r\rangle M_n, \quad r=1,2,\ldots,n.
\end{cases}
\end{equation}
Relation (\ref{eq:chain}) is a consequence of  (\ref{eq:ideals-definition-via-diagonal}) and
\begin{equation*}
\label{eq:rm1-r-rm1}
    \langle 0,1,r-1\rangle=
    \langle 0,1,r\rangle 
    \langle 0,1,r-1\rangle, \quad r=2,3,\ldots,n.
\end{equation*}
In (\ref{eq:chain}), we wrote \textit{strict} inclusions as (\ref{eq:ideals-definition-via-diagonal}) and (\ref{eq:rank-property}) give $\mathrm{rnk}(x)\le r$ for any $x\in I_r$, so
\begin{equation}
\label{eq:strict-inclusion}
\langle 0,1,r\rangle\notin I_{r-1},\quad r=1,2,\ldots,n. 
\end{equation}
$I_r=\{y\in M_n:\mathrm{rnk}(y)\le r\}$ follows easily from (\ref{eq:chain})--(\ref{eq:strict-inclusion}). 
Finally, \textit{any} ideal $I$ is the union of some $I_r$ because $I=\cup_{x\in I} M_nxM_n$. Thus, by (\ref{eq:chain}), all ideas are principal.
}
\end{proof}
\edit{
\begin{example}
$M_2$ has three ideals. By the second expression in  (\ref{eq:ideals-given-by}), they are given by
\begin{equation*}
    I_0=\{\zero\}, \quad I_1=\{\zero,e,f,a,b\},\quad I_2=\{\zero,e,f,a,b,\one\}=M_2,
\end{equation*}
where we used the notation of Example~\ref{ex:r2} and Proposition~\ref{prop:brandt} for the monoid elements. Note that Theorem~\ref{th:green}, (\ref{eq:ideals-definition-via-diagonal}), $e=\langle 0,1,1\rangle$, and $\one=\langle 0,1,2\rangle$ additionally imply
\begin{equation*}
    I_0=J_\zero, \quad I_1=M_2eM_2=I_0\cup J_e,\quad I_2=M_2\one M_2=I_1\cup J_\one.
\end{equation*}
\end{example}
}

\section{On the numbers of zero and nonzero products}
\label{section: numer-of-zero-products}

The set $S_2\setminus \{\zero\}$, which has $4$ elements, gives rise to $4^2=16$ possible products. The last four rows and columns of the table in Proposition~\ref{prop:brandt} tell us that $8$ of these products are zero and $8$ are nonzero. This case ($n=2$) seems to be the only one in which the numbers of zero and nonzero products are equal; e.g. $S_3\setminus \{\zero\}$, which has 13 elements, gives rise to 91 nonzero products, about $r(3)=53.9\%$ of the total. We have investigated the more general ratio $r(n)$ in the following manner.

Let  $\psi(n)$ be the number of nonzero products arising from the set $M_n\setminus\{\zero\}$. By (\ref{eq:m-definition}) and  Corollary~\ref{corr:nonzero}, $\psi(n)$ is given by
\begin{equation}
\label{eq:psi-sextuple}
\begin{split}
\psi(n)=\sum_{d=1-n}^{n-1}\,
&\sum_{k=1-\min(0,d)}^{n-\max(0,d)}\,
\sum_{m=k}^{n-\max(0,d)}\\
&\sum_{d'=1-n}^{n-1}\,
\sum_{k'=1-\min(0,d')}^{\min\left(m+d,n-\max(0,d')\right)}\,
\sum_{m'=\max(k',k+d)}^{n-\max(0,d')}1.
\end{split}
\end{equation}
Now $2\abs{S_n}-1$ of these nonzero products involve $\one$, so that $S_n\setminus \{\zero\}=M_n\setminus\{\zero,\one\}$  gives rise to  $\psi(n)-2\abs{S_n}+1$ nonzero products. As the total number of products is $(\abs{S_n}-1)^2$, the ratio $r(n)$ is
\begin{equation}
\label{eq:ratio}
    r(n)=\frac{\psi(n)-2\abs{S_n}+1}{(\abs{S_n}-1)^2}.
\end{equation}
Although were not able to analytically evaluate the sextuple sum giving $\psi(n)$, numerical  evaluation and use of the On-Line Encyclopedia of Integer Sequences (OEIS) suggested a connection to the so-called Polynexus numbers \edit{(sequence A083200 of the OEIS \cite{oeis})}, as stated in the following conjecture.
\begin{conjecture}
\label{conjecture:psi}
For $n\in\mathbb{N}$, $\psi(n)$ equals the $(n+1)$th Polynexus number of order $7$ \cite{oeis}, so that 
\begin{equation}
\label{eq:psi-polynexus}
    \psi(n)=\frac{(n+1)^7-n^7-(n+1)^3+n^3}{120}.
\end{equation}
\end{conjecture}
This conjecture was verified for all $n\le 70$. Provided that it is true, (\ref{eq:ratio}), (\ref{eq:order}), and (\ref{eq:psi-polynexus}) give a closed-form expression for $r(n)$: 

\begin{conditional proposition}
\label{cond:conditional}
If Conjecture~\ref{conjecture:psi} is true, then the ratio $r(n)$ of nonzero products in $S_n\setminus\{\zero\}$ to the total number of products in $S_n\setminus\{\zero\}$ is given by
\begin{equation}
    r(n)=\frac{3(7n^5+28n^4+63n^3+18n^2-84n-120)}{10(n-1)(2n^2+5n+6)^2}.
\end{equation}
Accordingly, as $n\to\infty$, 
\begin{equation}
    r(n)=\frac{21}{40}+\frac{147}{160}\,\frac{1}{n^2}+O\left(\frac{1}{n^3}\right).
\end{equation}
\end{conditional proposition}
Thus when $n$ is large, $r(n)$ approaches  $52.5\%=\frac{21}{40}$ monotonically from above. Numerical results are given in Figure~\ref{fig:ratio}, where we observe a peak at $n=4$.

\begin{figure}[h]
    \centering
    \includegraphics[width=.55\textwidth]{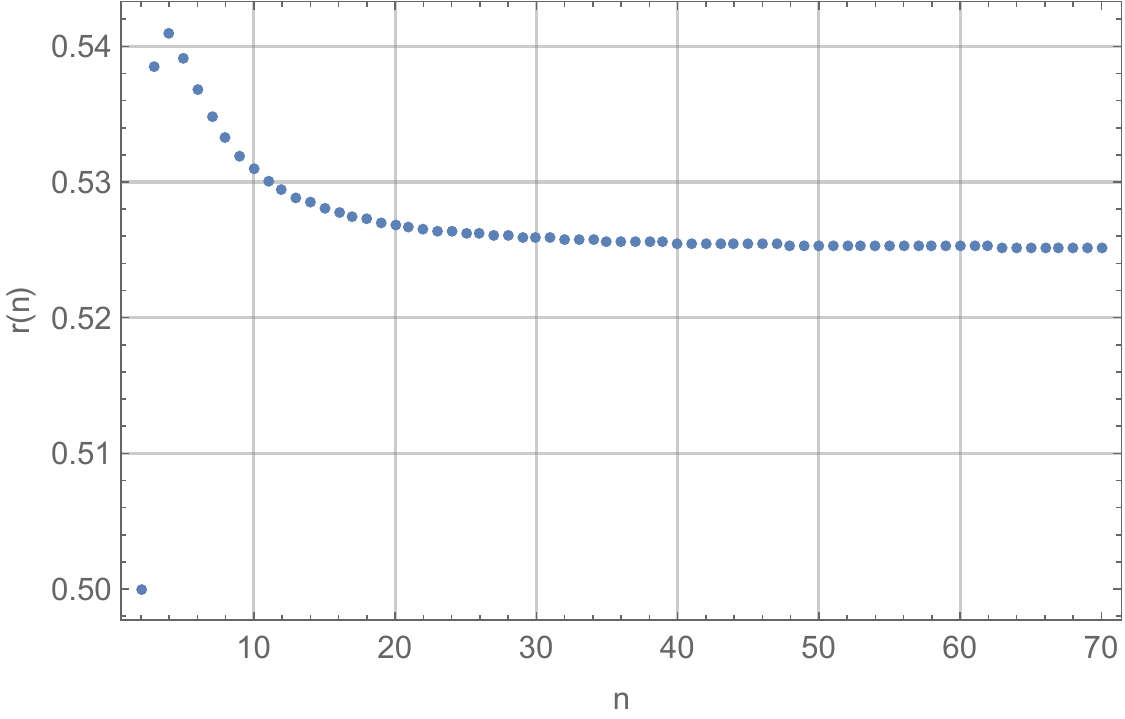}
    \caption{The ratio $r(n)$ for $n=2,3,\ldots,70$}
    \label{fig:ratio}
\end{figure}

\begin{remark}
The three sums in the second line of (\ref{eq:psi-sextuple}) form a triple sum which depends on $n$, $d$, $k$, and $m$. Given an $x=\langle d, k,m\rangle\in M_n\setminus\{\zero\}$, the triple sum is the number of $y=\langle d', k',m'\rangle\in M_n\setminus\{\zero\}$ for which $xy\ne \zero$.
Conversely, if we are given a $y$, it is easy to formulate a similar triple sum yielding the number of $x$ for which $xy\ne \zero$.
\end{remark}

\section{Submonoids and Subsemigroups}
\label{section:submonoids}

This section examines certain subsemigroups of $M_n$. To begin with, 
it can be verified from Theorem \ref{th:multiplication-finite} that, for $n\ge 3$,  the following subset of $M_n$
\begin{equation*}
    \{\zero\}\cup \{\langle d,k,m\rangle\in M_n\setminus\{\zero\}: m< n-\max(0,d)\}
\end{equation*}
is closed under multiplication. This set consists of $n\times n$ matrices whose elements in the $n$'th row and the $n$'th column are all zero. Thus closure alternatively follows from the general discussions in Section~\ref{section:notation} and amounts to 
\begin{equation}
    M_j\sub M_n,\quad 2\le j< n,
\end{equation}
where $\sub$ means ``is a subsemigroup of''. 
Note that the $\one$ of $M_n$ does not belong to $M_j$ ($M_j$ has a different $\one$ than does $M_n$) so that the monoids $M_2,\ldots, M_{n-1}$ are \textit{not} submonoids of $M_n$.

Next consider the subsets of $M_n$ that are \edit{described in Table~\ref{tab:Mn-subsets}.}

\begin{table}[!h]
\centering
\begin{tabular}{ |c| c| l|}
\hline
symbol for & matrices of $M_n$ representing & restrictions on\\
subset of $M_n$& nonzero elements& $d$, $k$, $m$\\
\hline
\hline
$UT_n$& upper triangular & $d\ge 0$\\
 \hline
 $SUT_n$& strictly upper triangular & $d> 0$\\
 \hline
 $UF_n$& upper full & $d\ge 0$, $k=1$, and $m=n-d$
 \\
 \hline
 $SUF_n$& strictly upper full & $d> 0$, $k=1$, and $m=n-d$
 \\
 \hline
 $LT_n$& lower triangular & $d\le 0$\\
 \hline
 $SLT_n$& strictly lower triangular & $d< 0$\\
 \hline
 $LF_n$& lower full & $d\le 0$, $k=1-d$, and $m=n$
 \\
 \hline
 $SLF_n$& strictly lower full & $d< 0$, $k=1-d$, and $m=n$
 \\
 \hline
 $D_n$& diagonal & $d=0$
 \\
 \hline
\end{tabular}
\caption{\edit{Subsets of $M_n$ which, by Theorem~\ref{th:subsemigroups}, are subsemigroups of $M_n$}}
\label{tab:Mn-subsets}
\end{table}

Each subset named in our table's first column consists of $\zero$ together with the elements $\langle d, k, m\rangle$ of $M_n\setminus\{\zero\}$ that satisfy the restrictions in the third column. As an example,
\begin{equation}
\label{eq:sutn-definition}
\begin{split}
SUT_n&=\{\zero\}\cup \{\langle d, k, m\rangle \in M_n\setminus\{\zero\}: d\ge 1\}\\&=\{\zero\}\cup\{\langle d,k,m\rangle:\  d,k, m\in\mathbb{N}; \ k \le m\le n-d\},
\end{split}
\end{equation}
where the second expression follows from (\ref{eq:m-definition}). As another example, 
\begin{equation}
SLF_n=\{\zero\}\cup \{\langle d, 1-d, n\rangle \in M_n\setminus\{\zero\}: d< 0\}.
\end{equation} 
Any nonzero element of $M_n$ corresponds to a matrix with only one nonzero diagonal. The term ``full'' appearing in the table means that \textit{all} entries of this diagonal are nonzero (and equal to $1$). The terms upper/lower triangular, strictly triangular, etc. are used in the usual sense of matrix analysis.

\begin{theorem} 
\label{th:subsemigroups}
All nine sets written in \edit{Table~\ref{tab:Mn-subsets}} are subsemigroups of $M_n$. The semigroups $UT_n$, $UF_n$, $LT_n$, $LF_n$, and $D_n$ contain $\one$ and are submonoids of $M_n$; removing $\one$ from any of these five sets gives a subsemigroup of $S_n=M_n\setminus\{\one\}$. Finally, $SUT_n$, $SUF_n$, $SLT_n$, and $SLF_n$ consist solely of nilpotents, whose indexes are given by (\ref{eq:index-of-nilpotent}). 
\end{theorem}
\begin{proof}
The first statement is evident from Theorem \ref{th:multiplication-finite} and the table's third column, or from elementary matrix theory and the second column. The second statement follows from $\one=\langle 0,1,n\rangle$ and Corollary \ref{corr:S-semigroup}. The third is a consequence of the table's third column and Theorem~\ref{th:idempotent-nilpotent}.
\end{proof}
\edit{We now let $A^T$ denote the set of transposes (see Proposition \ref{prop:transpose}) of the subset $A$, viz.,
\begin{equation}
    A^T=\{x^T:x\in A\subseteq M_n\}.
\end{equation}
When $A$ is a subsemigroup of $M_n$, it follows from $(xy)^T=y^Tx^T$ that $A$ and $A^T$ are anti-isomorphic.} 

Further properties of the above-discussed subsemigroups are given in the proposition that follows.
\begin{proposition}
\edit{Table~\ref{tab:properties-subsemigroups}} gives the transposes $A^T$ of semigroups $A$, states whether $A$ and $A^T$ are commutative or not, and gives semigroup orders.

\begin{table}[!h]
\begin{tabular}{|c|c|c|l|}
\hline
subsemigroup $A$ & transpose $A^T$ & $A$ and $A^T$&order of $A$ and $A^T$\\
of $M_n$ & of $A$ is & are & {}\\
\hline
\hline
$UT_n$ & $LT_n$ & noncommutative& $1+\frac{1}{6}n(n+1)(n+2)$\\
\hline
$SUT_n$ & $SLT_n$ & noncommutative& $1+\frac{1}{6}(n-1)n(n+1)$\\
\hline
$UF_n$ & $LF_n$ & commutative& $1+n$
\\
\hline
$SUF_n$ & $SLF_n$ & commutative& $n$
\\
\hline
$D_n$ & $D_n$ & commutative & $1+\frac{1}{2}n(n+1)$
\\
\hline
\end{tabular}
\centering
\caption{\edit{Transposes of the semigroups in Theorem~\ref{th:subsemigroups}, together with their orders}}
\label{tab:properties-subsemigroups}
\end{table}
\end{proposition}

\begin{proof}
The orders can be found by direct summation, similarly to the proofs of Propositions \ref{prop:order} and \ref{prop:band}.
\end{proof}

It is evident that
\begin{equation}
SUF_n\sub UF_n\sub UT_n \sub M_n, \quad
SUT_n\sub UT_n\sub M_n,
\end{equation}
and that the same relations hold when $U$ is replaced by $L$. It is also evident that
\begin{equation}
D_n=UT_n\cap LT_n=UF_n\cap LF_n.
\end{equation}

We proceed to give some additional properties of $D_n$, $SUF_n$, $SLF_n$, and $SUT_n$.

\subsection{Idempotents of \texorpdfstring{$M_n$}{Mn}; diagonal matrices}
\label{section:mdn}
\edit{As already discussed, $D_n$ is the semilattice of idempotents of $M_n$ and has order $\frac{n^2}{2}+\frac{n}{2}+1$. Multiplication in $D_n$ is given by
\begin{equation}
\langle 0,k,m\rangle \langle 0,k',m'\rangle=
\begin{cases}
\left\langle0,\max(k,k'),\min(m,m')\right\rangle,\quad \max(k,k')\le  \min(m,m'),\\
\zero,\quad \max(k,k')>\min(m,m').
\end{cases}
\end{equation}
and corresponds to the intersection of the corresponding diagonal vectors.}

\subsection{Shift matrices and subsemigroups they generate} 
The monoid elements $\langle 1,1,n-1\rangle$ and $\langle-1,2,n\rangle$ represent the usual backward and forward shift matrices \cite{HornJohnson}, which are transposes of one another. Their products are given by 
\begin{equation}
\langle 1,1,n-1\rangle\langle -1,2,n\rangle =\langle 0,1,n-1\rangle\in D_n
\end{equation}
and
\begin{equation}
\langle -1,2,n)\langle 1,1,n-1\rangle=\langle 0,2,n\rangle\in D_n.
\end{equation}
 
Corollary \ref{corr:powers} gives the respective powers as
\begin{equation}
\label{eq:shift-upper}
\langle 1,1,n-1\rangle^j=
\begin{cases}
\langle j,1,n-j\rangle,\ j=1,2,\ldots,n-1,\\
\zero,\ j=n,
\end{cases}
\end{equation}
and
\begin{equation}
\label{eq:shift-lower}
\langle -1,2,n\rangle^j=
\begin{cases}
\langle -j,1+j,n\rangle,\ j=1,2,\ldots,n-1,\\
\zero,\ j=n.
\end{cases}
\end{equation}
The powers in (\ref{eq:shift-upper}) and (\ref{eq:shift-lower}) are precisely the elements of $SUF_n$ and $SLF_n$, respectively; see the table in Theorem~\ref{th:subsemigroups}.  We have thus reached the following conclusion. 

\begin{proposition}
\label{prop:monogenic}
$SUF_n$ is the monogenic subsemigroup generated by $\langle 1,1,n-1\rangle$, and $SLF_n=SUF_n^T$ is the monogenic subsemigroup generated by $\langle -1,2,n\rangle=\langle 1,1,n-1\rangle^T$.
\end{proposition}

Proposition \ref{prop:monogenic} and eqns.   (\ref{eq:shift-upper}) and (\ref{eq:shift-lower}) 
can be recognized as well-known properties of the two shift matrices; in fact, these properties make clear why the set of upper (or lower) Toeplitz matrices of size $n$ form a commutative algebra \cite{HornJohnson}.

\subsection{Minimal generating set for \texorpdfstring{$SUT_n$}{SUTn}; rank of \texorpdfstring{$SUT_n$}{SUTn}}

 The semigroup $SUT_n$ was defined in (\ref{eq:sutn-definition}). The theorem that follows provides a minimal generating set (called $A_n$) for $SUT_n$, gives $\mathrm{rank}(SUT_n)$ \edit{(see Section~\ref{section:notation})}, and uses the triplet representation to find (non-unique) closed-form expressions for all $x\in SUT_n$. The generating set $A_n$ consists of all matrices of $SUT_n$ whose ones lie on the diagonal $d=1$. 

\begin{theorem}
\label{th:minimal-rank-sutn}
The subset $A_n$ of $SUT_n$ defined by
\begin{equation}
\label{eq:an-definition}
    A_n=\{\langle 1,k,m\rangle: 1\le k\le m \le n-1 \}
\end{equation}
is a minimal generating set for $SUT_n$, with
\begin{equation}
\label{eq:sutn-rank}
    \mathrm{rank}\left(SUT_n\right)=\abs{A_n}=\frac{1}{2}n(n-1).
\end{equation}
In fact, any $x\in SUT_n$ can be found explicitly as a power of some $y\in A_n$ via
\begin{equation}
\label{eq:sutn-explicit}
    x=
    \begin{cases}
     \langle 1,1,1\rangle^2,\quad x=\zero,\\
    \langle 1,k,m+d-1 \rangle^d,\quad x=\langle d,k,m\rangle\in SUT_n\setminus \{\zero\}. 
    \end{cases}
\end{equation}
\end{theorem}
\begin{proof}
By Theorem~\ref{th:idempotent-nilpotent}, $\langle 1,1,1 \rangle$ is a square root of~$\zero$; and by Theorem~\ref{th:roots}, $\langle 1,k,m+d-1 \rangle$ is the (unique in $M_n$) $d$th root  of $x=\langle d,k,m\rangle\in SUT_n$. Therefore  (\ref{eq:sutn-explicit}) holds. Since both $\langle 1,1,1 \rangle$ and $\langle 1,k,m+d-1 \rangle$ belong to $A_n$, (\ref{eq:sutn-explicit}) explicitly demonstrates that $A_n$ generates $SUT_n$.

We now turn to minimality and rank. If $x_i=\langle d_i,k_i,m_i\rangle\in SUT_n\setminus\{\zero\}$, then $\sum_{i=1}^p d_i\ge p$ by (\ref{eq:sutn-definition}); eqn. (\ref{eq:an-definition}) and Theorem~\ref{th:multiplication-finite} then show that the product $\prod_{i=1}^p x_i$ can belong to $A_n$ only if $p=1$ and $x_1 \in A_n$.  In other words, any $y\in A_n$ can be written as a finite product of elements of $SUT_n$ only in a trivial manner (i.e., the product can only have one factor, equal to $y$), implying that \textit{any} generating set must contain this $y$. We have thus shown that
\begin{equation*}
    (B_n\subseteq SUT_n \mathrm{\ and \ } B_n \mathrm{\ generates\ } SUT_n)\implies A_n\subseteq B_n. 
\end{equation*}

It follows that no proper subset of $A_n$ can generate $SUT_n$, proving minimality of the generating set $A_n$. It also follows that $\abs{A_n}\le \abs{B_n}$ for any $B_n$ generating $SUT_n$, so that the rank of $SUT_n$ equals the cardinality of its generating set $A_n$. Computing $\abs{A_n}$ via (\ref{eq:an-definition}), we arrive at (\ref{eq:sutn-rank}).
\end{proof}
By taking transposes, one can immediately re-state Theorem~\ref{th:minimal-rank-sutn} so as to apply to $SLT_n=SUT_n^T$. \edit{In this case, the generating set is $A_n^T$, where
\begin{equation}
\label{eq:an-transpose-definition}
    A_n^T=\{\langle -1,k,m\rangle: 2\le k\le m \le n \}
\end{equation}}

\subsection{More subsemigroups of \texorpdfstring{$M_n$}{Mn}}
\label{sec:further-submonoids}
We close Section~\ref{section:submonoids} by listing some additional subsemigroups of $M_n$ that can be studied by specializing the triplet representation developed in this paper. The said subsemigroups are obtained by adjoining $\{\zero\}$ to any one of the following sets, 

\begin{itemize}
    \item The matrices of $M_n$ with $d=$integer multiple of $d_0$, where $1\le d_0\le n-1$.
    \item The matrices of $M_n$ with $d\ge d_0$, where $1\le d_0\le n-1$.
    \item The matrices of $M_n$ whose first row and first column are zero.
    \item The matrices of $M_n$ whose first row and last column are zero.
    \item The matrices of $M_n$ that have at most $j$ nonzero elements, where $2\le j\le n$ (see Example~\ref{ex:mequalsk-example} for the case $j=1$).
\end{itemize}

\section{Generating systems for \texorpdfstring{$S_n$}{Sn} and \texorpdfstring{$M_n$}{Mn}} 

\edit{
Based on our findings in Section~\ref{section:submonoids}, the present section determines several generating sets for $S_n$ and $M_n$, both as semigroups and as inverse semigroups. Since $M_n=SUT_n\cup D_n\cup SLT_n$, $M_n$ is generated as a semigroup by $A_n\cup D_n \cup A_n^T$ (see Theorem~\ref{th:minimal-rank-sutn} and (\ref{eq:an-transpose-definition})). In fact, $A_n\cup\{\one\}\cup A_n^T$ suffices:
\begin{proposition}\label{prop:gen-of-Sn}
$S_n$ is generated by $A_n\cup A_n^T$, and 
$M_n$ is generated by $A_n\cup\{\one\}\cup A_n^T$. 
\end{proposition}
\begin{proof}
    For the assertion pertaining to $S_n=M_n\setminus\{\one\}$, it suffices to show that $A_n\cup A_n^T$ generates $D_n\setminus\{\one\}$. This is apparent from (\ref{eq:an-definition}), (\ref{eq:an-transpose-definition}), and the equalities
    \begin{equation*}
        \langle 0,k,m\rangle=\langle 1,k,m\rangle
        \langle -1,k+1,m+1\rangle, \quad 1\le k\le m\le n-1,
    \end{equation*}
    \begin{equation*}
        \langle 0,k,n\rangle=\langle -1,k,n\rangle
        \langle 1,k-1,n-1\rangle, \quad 2\le k\le n,
    \end{equation*}
    which are consequences of Theorem~\ref{th:multiplication-finite}. By Proposition~\ref{prop:not-a-group}, any set generating $M_n$ must further include $\one$. Thus $M_n$ is generated by $A_n\cup\{\one\}\cup A_n^T$.
\end{proof}
\begin{corollary}
    $S_n$ is generated, as an inverse semigroup, by $A_n$. The monoid $M_n$ is generated, as an inverse semigroup, by $A_n\cup\{\one\}$.
\end{corollary}
}
\edit{One might think $A_n\cup A_n^T$ is \textit{minimal} (as a semigroup-generating set for $S_n$). By the theorem that follows, this is far from true; for any $n$, two elements are enough,} \edit{a stark difference compared to the $\frac{1}{2}n(n-1)$ elements needed to generate $SUT_n$.}

\begin{theorem}
\label{th:minimal-generating-set-for-mn}
\edit{
    The subset $\{p,q\}$ of $S_n$, where
    \begin{align}
        \langle 1, 1, n-1\rangle = p = q^T = \langle -1, 2, n\rangle^T
    \end{align}
    is a minimal generating set for $S_n$.}
\end{theorem}
\begin{proof}
\edit{
    By Proposition \ref{prop:gen-of-Sn}, all we need to prove is that $\{p,q\}$ generates the sets $A_n$ and $A_n^T$. For any $1\le k\le m\le n-1$ we notice that
    \begin{equation*}
        \langle 1, k, m\rangle = p^{n-m} q^{n-m+k-1} p^k
    \end{equation*}
    which follows from Theorem~\ref{th:multiplication-finite} and Corollary~\ref{corr:powers}. This proves that we can generate $A_n$; as $\{p,q\} = \{p,q\}^T$, it also proves that we can generate $A_n^T$.
    Minimality follows from $p^j \ne q$ and $q^j \ne p$ for all $j\in\mathbb{N}$.
}
\end{proof}
\begin{corollary}
    \edit{
    The subset $\{p,q,\one\}$ of $M_n$ is a minimal generating set for $M_n$.
    }
\end{corollary}
\begin{proof}
\edit{$\one$ cannot be generated by $\{p,q\}$ (Proposition~\ref{prop:not-a-group}), so the assertion follows immediately from  Theorem~\ref{th:minimal-generating-set-for-mn} and $M_n=S_n\cup \{\one\}$.}
\end{proof}

\section{Graphical Interpretations}
\label{section:graphical}

This section examines the rook $n$-diagrams associated with $M_n$. As we already did in Figures~\ref{fig:rook} and \ref{fig:poi}, we assume that the two rows of vertices form an $n\times 2$ orthogonal grid.

A rook $n$-diagram represents an element of $M_n$ iff its lines (edges) are parallel, equispaced, and uninterrupted,  as in Figure~\ref{fig:mn}.  Note that the said requirements guarantee non-intersecting lines, so that $M_n\sub \mathcal{POI}_n$. 

\begin{figure}[h]
  \centering
  \includegraphics[width=.95\textwidth]{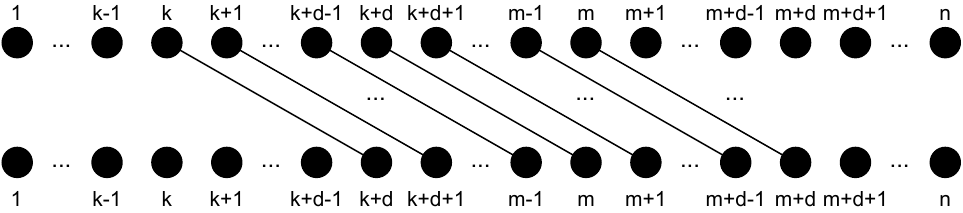}
  \caption{Rook $n$-diagram representation of the element $\langle d,k,m \rangle$ belonging to 
  $M_n\setminus\{\zero\}$}
  \label{fig:mn}
\end{figure}

Multiplication of two elements in $M_n$ is done graphically in the usual manner, as illustrated in Figure~\ref{fig:multiplication}.

\begin{figure}[h]
  \centering
  \includegraphics[width=\textwidth]{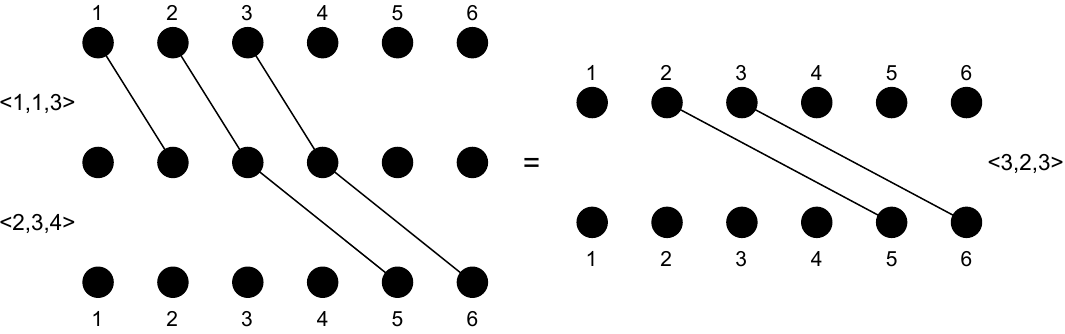}
  \caption{Graphical multiplication in $M_n$ via rook $n$-diagrams: Here, we depict $\langle 1,1,3 \rangle \langle 2,3,4 \rangle = \langle 3,2,3 \rangle$}
  \label{fig:multiplication}
\end{figure}

It is evident that many of our previous results can be depicted and explained graphically, and that graphical methods can help one discover further properties of $M_n$ and its subsemigroups. We thus proceed to give some indicative graphical interpretations (there are many more). In the discussions that follow, we conveniently refer to the parameter $d$ as the (positive, negative, or zero) \textsl{inclination} of the rook $n$-diagram representing $\langle d,k,m\rangle$.

\begin{itemize}
    \item It is apparent (e.g. from Figure~\ref{fig:mn}) that $\one$ is the only element of $M_n$ that has $n$ lines. This fact and the graphical interpretation of multiplication  make it easy to understand why the ``if'' part of Proposition~\ref{prop:not-a-group} ($xy=\one\implies x=y=\one$) is true: The number of lines in the product of two elements is at most the minimum number of the two elements' lines. Thus if $xy=\one$, then both $x$ and $y$ must have $n$ lines, so that $x=y=\one$.

    \item The subsets $D_n$ and $SUT_n$ correspond to a zero or positive inclination, respectively. It is apparent that  multiplication again leads to a zero or positive inclination, explaining why $D_n$ and $SUT_n$ are subsemigroups. It is also apparent why the elements of $D_n$ are the idempotents of $M_n$.
    
    \item The elements of the subset $UF_n$ are those whose first (leftmost) line starts at vertex $1$ of the top row, and whose last (rightmost) line ends at vertex~$n$ of the bottom row. Evidently multiplication retains these properties, so that $UF_n$ is a subsemigroup. The subsemigroup's $n$ elements correspond to the $n$ possible inclinations $d=0,1,\ldots,n-1$.

    \item Figure~\ref{fig:multiplication} helps one understand why multiplication adds inclinations and why taking the $j$th power multiplies the inclination by $j$, as expressed by  (\ref{eq:ddoubleprime-finite}) and (\ref{eq:powers-2}).  We can further understand formula (\ref{eq:index-of-nilpotent}) for  the index of a nonzero nilpotent: Let $x=\langle d,k,m\rangle$ with $d>0$. The first line of $x^j$ starts at (top) vertex $k$ and ends at (bottom) vertex $k+jd$. However, the last line of $x^j$ cannot end after vertex $m+d$. This means that the index $\ell$ of $x$ is the smallest $j$ such that $x^j$'s first line would have been after $m+d$, i.e. the minimum $j$ for which $k+jd\geq m+d+1$. When $d<0$, we can find the index in an analogous manner.
    
    \item In a similar manner, we can perceive when the  product of two elements is zero: The first line of $\langle d,k,m\rangle\langle d',k',m'\rangle$ ends on vertex $k+d+d'$ or to the right of that vertex, but it cannot end to the right of vertex $m'+d'$. Therefore $k+d+d'>m'+d'$ is a sufficient condition for the product to be $\zero$, as asserted by Corollary~\ref{corr:nonzero}.
    
    \item Much of Theorem \ref{th:roots} can be explained in the following manner. The inclination of $\langle d',k',m'\rangle^j$ is $jd'$. Thus $x = \langle d,k,m\rangle$ can have a $j$th root only if $d$ is a multiple of $j$. Assume now that $d>0$. The first line of the $j$th root of $\langle d,k,m\rangle$, must start at $k$ and its last line must end at $m+d$. This means that $y = \langle d',k',m'\rangle$ and $y^j = x$ imply $k' = k$ and $m' + d' = m + d$, as illustrated in Figure~\ref{fig:roots}.  
    \begin{figure}[h]
        \centering
        \includegraphics[width=.4\textwidth]{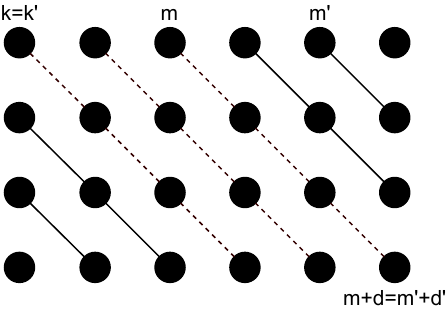}
        \caption{Graphical interpretation of $\langle d',k',m'\rangle^j = \langle d,k,m\rangle$. Each row shows $\langle d',k',m'\rangle$, while the dashed lines show $\langle d,k,m\rangle$}
        \label{fig:roots}
    \end{figure}
\end{itemize}

\section{Summary}\label{sec13}

In this paper, we defined a \edit{noncommutative, inverse} monoid $M_n$; showed that the nonzero elements of $M_n$ can be represented as $\langle d,k,m \rangle$, where $d$, $k$, and $m$ are properly defined integers; and used this ``triplet representation'' to study numerous properties of $M_n$. Both $M_n$ and the set $S_n=M_n\setminus \{\one\}$ (where $\one$ is the monoid identity, corresponding to the $n \times n$ identity matrix) are \edit{inverse} subsemigroups of the well-known rook monoid $R_n$, also known as the symmetric inverse semigroup. The monoid $M_n$ consists of those elements of $R_n$ whose ones lie on a single diagonal and form an uninterrupted block (i.e., no zero lies between any two ones). Equivalently, the rook-$n$ diagram of any monoid element has the form shown in Figure~\ref{fig:mn}. 

In Sections~\ref{section:infinite} and \ref{section:finite}, we regarded $M_n$ as a subsemigroup of a certain semigroup $S_\infty$ of countably infinite order; this helped us find the triplet representation of the product of two monoid elements (Theorem~\ref{th:multiplication-finite}).  Our results on $M_n$ include:
\begin{itemize}
\item A simple formula for powers (Corollary~\ref{corr:powers}).
\item A proof that any element is idempotent or nilpotent (Theorem~\ref{th:idempotent-nilpotent}). Proposition~\ref{prop:band} further gives the numbers of idempotents and nilpotents.
\item A formula (in Theorem~\ref{th:idempotent-nilpotent}) for the index $\ell$ of any nonzero nilpotent  in terms of $d$, $k$, and $m$. 
\item A proof that the $j$th root of any given nonzero element either does not exist; or does exist, is unique, and has the triplet representation given explicitly in Theorem~\ref{th:roots}.

\item \edit{Green's relations (Theorem~\ref{th:green}), enabling a determination of all ideals (Theorem~\ref{th:ideals}).}

\item An investigation of the ($n$-dependent) number of nonzero products in $S_n~\setminus~\{\zero\}$; this number is $50\%$ of the total number of products when $n=2$, and appears to be slightly larger than  $50\%$ for all other $n$ (Conditional Proposition~\ref{cond:conditional} and Figure~\ref{fig:ratio}). 

\item The use of rook-$n$ diagrams to give graphical interpretations of many of our results (Section~\ref{section:graphical}).  

\item An investigation of some subsemigroups of $M_n$ (Section~\ref{section:submonoids}). Two of these are the familiar monogenic subsemigroups generated by the well-known backward and forward shift matrices (Proposition~\ref{prop:monogenic}). For another subsemigroup---denoted by $SUT_n$ and consisting of all strictly upper triangular matrices of $M_n$---we found  a minimal generating set $A_n$, gave a (non-unique) explicit formula for any $x\in SUT_n$, and computed the rank of $SUT_n$ (Theorem~\ref{th:minimal-rank-sutn}). 

\item \edit{The surprising result (Theorem~\ref{th:minimal-generating-set-for-mn}) that, irrespective of $n$, two elements suffice to generate $S_n$.}

\item Section~\ref{sec:further-submonoids}, finally, lists more subsemigroups that can be investigated along similar lines. 
\end{itemize}

\backmatter

\bibliography{ref.bib}

\end{document}